\pdfoutput=1
\documentclass[11pt]{amsart}
\usepackage{amsmath,amssymb,amsthm}
\usepackage[margin=1.2in]{geometry}
\usepackage{booktabs}
\usepackage{adjustbox}
\usepackage{xcolor}
\usepackage[colorlinks=true,linkcolor=blue,citecolor=blue,urlcolor=blue]{hyperref}
\hypersetup{pdfauthor={Bernd Johannes Wuebben},
  pdftitle={Doubly isolated Batyrev mirror pairs and non-smoothable Calabi-Yau threefolds}}

\newtheorem{theorem}{Theorem}[section]
\newtheorem{proposition}[theorem]{Proposition}
\newtheorem{lemma}[theorem]{Lemma}
\newtheorem{corollary}[theorem]{Corollary}

\theoremstyle{definition}

\newtheorem{example}[theorem]{Example}
\newtheorem{question}[theorem]{Question}
\theoremstyle{remark}
\newtheorem{remark}[theorem]{Remark}

\newcommand{\ZZ}{\mathbb{Z}}
\newcommand{\RR}{\mathbb{R}}

\newcommand{\PP}{\mathbb{P}}
\newcommand{\FF}{\mathbb{F}}
\newcommand{\NN}{N}
\newcommand{\MM}{M}
\newcommand{\conv}{\operatorname{conv}}
\newcommand{\Cone}{\operatorname{Cone}}

\title[Doubly isolated Batyrev mirror pairs]{Doubly isolated Batyrev mirror
pairs\\ and non-smoothable Calabi--Yau threefolds}
\author{Bernd Johannes Wuebben}
\date{August 15, 2026}

\begin{document}

\begin{abstract}
We classify reflexive four-polytopes for which the generic
anticanonical hypersurfaces associated with the face fans of the polytope
and its polar both have at most isolated singularities. There are exactly
590 such polytopes up to lattice equivalence in the Kreuzer--Skarke
classification.
Their two-dimensional faces belong to twelve lattice-isomorphism
classes, all triangles, zonotopes, or reflexive polygons. The only
nonsmoothable local models are the cyclic quotients
$\frac13(1,1,1)$ and $\frac15(1,1,3)$ and the anticanonical cone over the
first Hirzebruch surface; no singularity has a positive-dimensional
reduced deformation space without a smoothing component. Within this
classification, the Hirzebruch cone occurs in a unique mirror pair,
associated with polytopes having 22 and 26 vertices. We exhibit this
pair explicitly: one hypersurface is nonsmoothable, while all
singularities of its mirror are locally smoothable. The crepant
resolutions have Hodge numbers $(20,26)$ and $(26,20)$. At the opposite
extreme, the self-dual 24-cell is the unique member of the classification
for which both hypersurfaces are smooth before crepant resolution.
\end{abstract}

\maketitle

\section{Introduction}\label{sec:intro}

Let $\Delta \subset \NN_\RR = \RR^4$ be a reflexive polytope all of whose edges
have lattice length one, let $\Delta^\circ \subset \MM_\RR$ be its polar dual,
and let $X \in |{-K_{\PP_\Delta}}|$ be a generic anticanonical hypersurface in
the Gorenstein Fano toric fourfold associated to the face fan of $\Delta$.
(Throughout, $\Delta$ is the \emph{fan} polytope, Batyrev's $\Delta^\ast$,
so our $\Delta^\circ$ is his $\Delta$; the same convention is in force in
\cite{paper1,paper2}.)  By
Batyrev \cite{Batyrev94}, $X$ is a Calabi--Yau threefold with Gorenstein
canonical singularities.  For each two-dimensional face
$F\preceq\Delta$ whose cone $C(F)$ is singular, $X$ has exactly
$\ell(F^\circ)$ points with analytic germ $C(F)$, and these are all its
singularities \cite[Prop.~3.1]{paper1}.  Here $\ell(F^\circ)$ denotes
the lattice length of the edge $F^\circ\preceq\Delta^\circ$ dual to $F$.

The Altmann--Corti--Filip--Petracci trichotomy, recalled as \cite[Thm.~2.2]{paper1},
classifies $C(F)$ by the combinatorics of $F$ into reduced-rigid
\textup{(R)}, deformable-but-non-smoothable \textup{(D)}, and smoothable
\textup{(S)}.  Here \emph{reduced-rigid} means that the reduced miniversal
base is a point; first-order deformations may still exist.  A single face of the first two types forces $X$ to admit no smoothing, and
\cite{paper2} studies when the type-\textup{(S)} configurations smooth.

Mirror symmetry, in the Batyrev picture, is the polar duality
$\Delta \leftrightarrow \Delta^\circ$, which for $4$-polytopes exchanges
two-dimensional faces with edges.  This note asks the mirror question: when
are $X$ and the mirror $X^\circ \subset \PP_{\Delta^\circ}$ \emph{both}
isolated-singular (we then call the pair \emph{doubly isolated}), and
which local models, in particular which
non-smoothable ones, can occur in that situation?

The entry point is exact and elementary.  For a two-dimensional face $F$
write $u_1, u_2 \in \MM$ for the primitive inner normals of the two facets of
$\Delta$ containing $F$, so $F^\circ = \conv\{u_1,u_2\}$, and let
$\MM_F = \MM \cap (\RR u_1 + \RR u_2)$.  We call the two-dimensional cone
$\Cone(u_1,u_2)$ the \emph{normal cone} of $F$ and its dual the
\emph{transverse cone} of $\Delta$ along $F$ (the local model of
$\Delta$ in the directions normal to $F$), and we say $\Delta$ is
\emph{transversally smooth} along $F$ when this cone is smooth.

\begin{lemma}\label{thm:intromult}
$\ell(F^\circ) = [\,\MM_F : \ZZ u_1 + \ZZ u_2\,]$, the multiplicity of the
normal cone $\Cone(u_1,u_2)$ of $F$ (equivalently, of its dual, the
transverse cone of $\Delta$ along $F$).  In particular $\ell(F^\circ)=1$ if
and only if the transverse cone is smooth.
\end{lemma}

The proof is given in Lemma~\ref{lem:mult};
Remark~\ref{rem:transverse} explains its transverse interpretation.

$X$ is isolated-singular exactly when $\Delta$ has unit edges, and
$X^\circ$ exactly when every two-dimensional face of $\Delta$ has
$\ell(F^\circ)=1$; a face with $\ell(F^\circ) = m \ge 2$ contributes to
$X^\circ$ a curve of transverse $A_{m-1}$ singularities
(Proposition~\ref{prop:corr}(2)).  We obtain:

\begin{corollary}\label{cor:introboth}
$X$ and $X^\circ$ are both isolated-singular if and only if $\Delta$ has unit
edges and every two-dimensional face of $\Delta$ is transversally smooth.  We
call such $\Delta$ \emph{both-sides unit}.  The condition is self-dual:
$\Delta$ is both-sides unit if and only if $\Delta^\circ$ is.
\end{corollary}

Our main results determine the both-sides unit polytopes and the local
models they carry by a finite computer-assisted classification.  The scan
covers all $473{,}800{,}776$ reflexive $4$-polytopes
\cite{KreuzerSkarke00}.  Reflexivity and absence of repetitions up to
lattice equivalence have been verified for every entry of the input
copy, so the published classification total establishes its completeness.
Every polytope returned by the scan is also re-verified in exact arithmetic.
The enumeration and the separate verification of its input are described
in \S\ref{sec:census}.

\begin{theorem}[classification of both-sides local models]\label{thm:class}
There are exactly $590$ both-sides
unit reflexive $4$-polytopes.  Every
two-dimensional face of every one of them is
\begin{itemize}
\item a triangle (the germ is a cyclic quotient point, or smooth), or
\item a zonotope (a centrally-symmetric unit-edge polygon, always of
type \textup{(S)}), or
\item a reflexive polygon: one of the five unit-edge polygons with exactly
one interior lattice point, i.e.\ the Fano polygons of the five
smooth toric del Pezzo surfaces.
\end{itemize}
All five reflexive polygons occur.  The complete face inventory is given in
Table~\textup{\ref{tab:occ}}.
\end{theorem}

\begin{theorem}[the non-smoothable germs of doubly isolated pairs]
\label{thm:dichot}
Exactly three of the $590$ both-sides unit
polytopes carry a non-smoothable
two-dimensional face.  Writing $[k,i]$ for a unit-edge polygon with $k$
edges and $i$ interior lattice points, the germs occurring are:
\begin{itemize}
\item $\frac13(1,1,1)$ (the cone over the $[3,1]$ triangle), six times on
one $9$-vertex polytope;
\item $\frac15(1,1,3)$ (the cone over the $[3,2]$ triangle), ten times on
a single simplex;
\item the anticanonical cone over $\FF_1$, exactly \emph{once}, on a
unique $22$-vertex polytope.
\end{itemize}
In particular no deformable-but-non-smoothable
\textup{(D)} face occurs on any both-sides unit polytope.
\end{theorem}

\begin{corollary}\label{cor:defonly}
Let $X$ be a unit-edge Batyrev
Calabi--Yau threefold with a singular point of
type \textup{(D)}, for instance the pentagon-cone threefolds of
\cite[Thms.~5.3, 5.8]{paper1}.  Then the Batyrev mirror $X^\circ$ is
non-isolated: it carries a curve of transverse $A_{m-1}$ singularities for
some $m \ge 2$.  The
deformable-but-non-smoothable phenomenon never occurs on a doubly isolated
mirror pair.
\end{corollary}

For the remaining non-quotient case, write $\mathrm{dP}_d$ for the smooth
toric del Pezzo surface of degree $d$, with $d=6,7$.

\begin{theorem}[the unique $\FF_1$ mirror pair]\label{thm:pair}
Up to lattice isomorphism there is exactly
one doubly isolated Batyrev mirror
pair whose singularities are not all quotient points or smoothable: the pair
$(\Delta_{\FF_1}, \Delta_{\FF_1}^\circ)$ of Theorem~\textup{\ref{thm:pairbody}},
with $22$ and $26$ vertices.  The generic anticanonical
$X \subset \PP_{\Delta_{\FF_1}}$ has $20$ ordinary double points, one
$\mathrm{dP}_6$-cone point, and one point whose germ is the anticanonical
cone over $\FF_1$; since that germ admits no smoothing at all, $X$
\emph{admits no smoothing} \cite[Cor.~4.3]{paper1}.  Its mirror $X^\circ$
has $26$ ordinary double
points, two $\mathrm{dP}_7$-cone points and two $\mathrm{dP}_6$-cone
points; every singular germ of $X^\circ$ is of type \textup{(S)}.  The maximal
projective crepant partial resolutions have
$(h^{1,1},h^{2,1}) = (20,26)$ and $(26,20)$ respectively.
\end{theorem}

Thus a nonsmoothable Calabi--Yau threefold can have an isolated-singular
Batyrev mirror with every singular point locally smoothable.  Every dual
edge on the mirror side has length one, so the single-face construction
of \cite{paper2} does not apply.  The $26$ diagonal-relation vectors
of its nodes have rank $18$ and two coloops: the coefficients of those
two vectors vanish in every relation.  Hence there is no relation with
every coefficient nonzero.  This nodal calculation alone does not
decide smoothability in the presence of the four del Pezzo-cone points.
The mixed necessity theorem of \cite{paper4} proves that the mirror is
also nonsmoothable, as discussed in
Remark~\ref{rem:node-subsystem}.

Theorems~\ref{thm:class}, \ref{thm:dichot} and~\ref{thm:pair} answer the
mirror question of
\cite[\S6.4]{paper1} for the non-smoothable local models: among them,
only the cyclic quotients and the $\FF_1$-cone,
including the germs of Gross's two exceptional surfaces $\PP^2$ and
$\FF_1$ \cite[Thm.~5.8]{Gross97}, admit doubly isolated mirror
realizations, and the $\FF_1$-cone does so exactly once in the entire toric
hypersurface landscape.

At the opposite extreme of the classification, exactly eight of the $590$ polytopes have all two-dimensional faces smooth,
so that $X$ itself is smooth; for exactly one of them the mirror is smooth
as well: the self-dual $24$-cell, whose hypersurface is the classical
self-mirror threefold with $(h^{1,1},h^{2,1}) = (20,20)$ \cite{Braun12}.
The smooth Batyrev mirror pair is thus unique in the entire
classification (Example~\ref{ex:24cell}).

\subsection*{The first asymmetric faces}
Call a polygon \emph{asymmetric} when it is neither a triangle nor
centrally symmetric.  No such face occurs
on a both-sides unit polytope with at most $18$ vertices, a range
containing $395{,}406{,}329$ reflexive $4$-polytopes
($83.5\%$ of the classification).  The first asymmetric face occurs at
$19$ vertices.  The complete inventory consists of the $\mathrm{dP}_7$
pentagon ($87$ faces, on polytopes with $19$ to $31$ vertices) and the
$\FF_1$ quadrilateral (once).  Section~\ref{sec:refuted} gives the first
example and explains why a bound on the number of vertices does not
respect polar duality.

All mathematical predicates and reported invariants in the scan, the exact
re-verification of all $590$ polytopes, and the mirror-pair calculation use
integer or rational arithmetic, including the cyclic ordering of planar
faces.  The programs build on those of
\cite{paper1} and are included as ancillary files;
\texttt{both\_sides\_census.py} reconstructs the finite face inventories,
counts, and Hodge numbers from the recorded polytopes.  The
code and the scan data are public at
\url{https://github.com/bwuebben/calabi-yau-smoothability}, where every
count in this note is recomputable from
\texttt{src/both\_sides\_census.py}.

\subsection*{Acknowledgements}
This note continues the series \cite{paper1,paper2}; the question of
smoothing canonical Calabi--Yau singularities was posed to the author by
Mark Gross many years ago.
The author used LLMs during the development of this work for literature exploration, editorial and coding assistance and as an adversarial reader.

\section{Preliminaries}\label{sec:prelim}

We keep the notation of \cite{paper1}.  For a lattice polygon $Q$ with unit
edges placed at height one, $C(Q)$ denotes the associated Gorenstein toric
affine threefold.  A unimodular triangle gives a smooth cone and is
excluded from the following trichotomy.  For every other unit-edge
polygon, $Q$ is of type \textup{(R)}, \textup{(D)} or \textup{(S)}
according as its edge-vector multiset has no proper nonempty zero-sum
sub-multiset, has one but no partition into opposite primitive pairs and
unimodular zero-sum triples, or has such a partition
\cite{Altmann97,CFP}.  Here a zero-sum triple $(a,b,c)$ is unimodular
when $|\det(a,b)|=1$.  Type \textup{(R)} cones are \emph{reduced-rigid}: their reduced miniversal
base is a point, although first-order deformations may exist.  Type \textup{(S)} polygons are
exactly those with singular smoothable cone; a single germ of type \textup{(R)} or
\textup{(D)}, one whose miniversal base has no smoothing component,
obstructs every global smoothing of $X$
\cite[Lemma~4.1, Cor.~4.3]{paper1}.  The notation $[k,i]$ records the number $k$ of edges and the number $i$
of interior lattice points of a unit-edge polygon; these two invariants
do not determine its lattice-isomorphism class.  The kind column of
Table~\ref{tab:occ} distinguishes the two classes with invariants $[4,1]$.

A
unit-edge polygon is a
\emph{zonotope} (a Minkowski sum of unit segments) if and only if it is
centrally symmetric, and every zonotope is of type \textup{(S)}.  A unit-edge
polygon with exactly one interior lattice point is reflexive.  Indeed,
translate that point to the origin and let the polygon have $k$ edges.
Pick's theorem gives normalized area $k$, while the fan from the origin
decomposes it into $k$ lattice triangles of positive integral normalized
areas, one over each unit edge.  Every one therefore has normalized area
one, so the origin has lattice distance one from every edge.  There are
five, the Fano polygons of $\PP^2$, $\FF_1$, $\PP^1{\times}\PP^1$,
$\mathrm{dP}_7$, $\mathrm{dP}_6$ (the subscript is the degree:
$\mathrm{dP}_7 = \mathrm{Bl}_2\PP^2$, $\mathrm{dP}_6 = \mathrm{Bl}_3\PP^2$),
the cone $C(Q)$ being the anticanonical
cone over the corresponding surface.  Of the five, the $\PP^2$ triangle and
the $\FF_1$ quadrilateral are reduced-rigid and the other three are of type
\textup{(S)} \cite[Rem.~2.6]{paper1}.  For a reflexive $4$-polytope $\Delta$ we write
$\Delta^\circ = \{u : \langle u,v\rangle \ge -1\ \forall v \in \Delta\}$; a
face $G \preceq \Delta$ has dual face $G^\circ \preceq \Delta^\circ$ with
$\dim G + \dim G^\circ = 3$.  By
\cite[Prop.~3.1]{paper1}, for $\Delta$ with unit edges the generic
$X \in |{-K}|$ has, over each two-dimensional face $F$ with $C(F)$
singular, exactly
$\ell(F^\circ)$ singular points with germ $C(F)$, and no other singular
points; and an edge of $\Delta$ of lattice length $m \ge 2$ produces a curve
of transverse $A_{m-1}$ singularities \cite[Rem.~3.2]{paper1}.  Throughout,
$\Delta \subset \NN_\RR$ is the \emph{fan} polytope, Batyrev's
$\Delta^\ast$, so our $\Delta^\circ$ is his $\Delta$; the same
convention is used in \cite{paper1,paper2}.

\section{The transverse-multiplicity identity}\label{sec:mult}

Let $F$ be a codimension-two face of a reflexive polytope
$\Delta \subset \NN_\RR$ of arbitrary dimension, and let
$u_1, u_2 \in \MM$ be the primitive inner
normals of the two facets of $\Delta$ containing $F$, so that
$\langle u_i, \cdot\rangle = -1$ on $F$ and $F^\circ = \conv\{u_1, u_2\}$.
\footnote{The $u_i$ are vertices of $\Delta^\circ$; \cite{paper1} and
the ancillary programs list the sign-flipped functionals taking the value
$+1$ on the facet.  Lattice lengths are unaffected.}
Put $\MM_F = \MM \cap (\RR u_1 + \RR u_2)$, a rank-two saturated sublattice
of $\MM$ containing $u_1, u_2$ as primitive vectors.

\begin{lemma}\label{lem:mult}
$\ell(F^\circ) = [\,\MM_F : \ZZ u_1 + \ZZ u_2\,]$.
\end{lemma}

\begin{proof}
Both sides depend only on the pair $(u_1,u_2)$ inside the rank-two lattice
$\MM_F$.  Choose a $\ZZ$-basis of $\MM_F$ in which $u_1 = (0,-1)$; this is
possible because $u_1$ is primitive.  Write $u_2 = (p,q)$; here $p \neq 0$, since $p = 0$ would force
$u_2 = \pm u_1$, impossible as $F^\circ = \conv\{u_1,u_2\}$ is an edge of
$\Delta^\circ$ and $0$ is interior to $\Delta^\circ$.  Then $p =
\det(u_1,u_2) = \pm[\,\MM_F : \ZZ u_1 + \ZZ u_2\,]$, and $\ell(F^\circ) =
\ell\bigl(\conv\{u_1,u_2\}\bigr) = \gcd(u_2 - u_1) = \gcd(p,\,q+1)$,
the two gcds agreeing because $\MM_F$ is saturated in $\MM$, so lattice
length in $\MM$ may be computed in the chosen basis of $\MM_F$.

Let $v$ be a vertex of $F$ and let $\bar v \in \NN_F := \NN/(\NN \cap
u_1^\perp \cap u_2^\perp)$ be its image; since $u_1,u_2$ vanish on
$\NN \cap u_1^\perp \cap u_2^\perp$ they descend to elements of
$\MM_F = \NN_F^\vee$, and $\bar v$ is a lattice point of $\NN_F \cong \ZZ^2$
with $\langle u_i, \bar v\rangle = -1$.  Writing $\bar v = (a,b)$ in the
basis of $\NN_F$ dual to the chosen basis of $\MM_F$,
$\langle u_1, \bar v\rangle = -b = -1$ gives $b = 1$, and
$\langle u_2, \bar v\rangle = pa + q = -1$ gives $p a = -1-q$, so
$p \mid (q+1)$.  Hence $\gcd(p, q+1) = |p|$, and
$\ell(F^\circ) = |p| = [\,\MM_F : \ZZ u_1 + \ZZ u_2\,]$.
\end{proof}

\begin{remark}\label{rem:transverse}
Lemma~\ref{lem:mult} is elementary and may be known; what matters here is
the reading it provides.  The proof works in every dimension for a
codimension-two face.  The right-hand side is the multiplicity of the
two-dimensional cone $\Cone(u_1,u_2)$, which is the normal cone of $F$ and
dually the transverse cone of $\Delta$ along $F$ (for two-dimensional cones
the multiplicities of a cone and its dual agree).  Thus \emph{the dual-edge
length counts exactly the transverse singularity multiplicity}, and
$\ell(F^\circ)=1$ is the statement that $\Delta$ is transversally smooth
along $F$.  (Transversal smoothness concerns only the directions normal to
$F$: the germ $C(F)$ itself may well be singular.)  The lattice-point
hypothesis on $v$ is essential: for an
arbitrary pair of primitive vectors the two sides can differ.
\end{remark}

Corollary~\ref{cor:introboth} is immediate: $X$ is isolated-singular iff
$\Delta$ has unit edges; $X^\circ$ is isolated-singular iff $\Delta^\circ$
has unit edges, i.e.\ every edge $F^\circ$ of $\Delta^\circ$ has
$\ell(F^\circ)=1$, i.e.\ by Lemma~\ref{lem:mult} every two-dimensional face
of $\Delta$ is transversally smooth.  Self-duality follows since the
conditions ``unit edges'' and ``all dual edges of length one'' are exchanged
by $\Delta \leftrightarrow \Delta^\circ$.

\section{The mirror correspondence of the singular strata}\label{sec:corr}

\begin{proposition}\label{prop:corr}
Let $\Delta$ be reflexive with unit edges and $F$ any two-dimensional
face.  Under polar duality $F$ corresponds to the edge
$F^\circ$ of $\Delta^\circ$, and:
\begin{enumerate}
\item if $C(F)$ is singular, $X$ has $\ell(F^\circ)$ isolated points of
type $C(F)$ over $F$ (and none over $F$ otherwise);
\item $X^\circ$ has, along the surface
$V(\sigma_{F^\circ}) \subset \PP_{\Delta^\circ}$ (the orbit closure of
the cone $\sigma_{F^\circ}$ over the edge $F^\circ$ in the face fan of
$\Delta^\circ$), a curve of transverse $A_{\ell(F^\circ)-1}$
singularities (and none if $\ell(F^\circ)=1$).
\end{enumerate}
In particular $X^\circ$ is isolated-singular if and only if
$\max_F \ell(F^\circ) = 1$, i.e.\ $\Delta$ is both-sides unit.
\end{proposition}

\begin{proof}
(1) is \cite[Prop.~3.1]{paper1}.  For (2), the edge $F^\circ$ of
$\Delta^\circ$ has lattice length $\ell(F^\circ)$; an edge of lattice length
$m$ produces on the generic anticanonical hypersurface a curve of transverse
$A_{m-1}$ singularities (\S\ref{sec:prelim}), and none when $m = 1$.  Hence
$X^\circ$ is isolated-singular iff every two-dimensional face of $\Delta$
has $\ell(F^\circ) = 1$.
\end{proof}

Thus the $\ell(F^\circ)$ singular points of $X$ over $F$ and the
$A_{\ell(F^\circ)-1}$-curve of $X^\circ$ correspond under polar duality,
and the two
threefolds are simultaneously isolated-singular exactly on the both-sides
unit polytopes.  Whenever $\Delta$ has \emph{any} face with
$\ell(F^\circ) \ge 2$ (the generic situation), the mirror is
non-isolated and its singularities leave the isolated framework of
\cite{Altmann97,CFP}.  Filip \cite{Filip25} constructs, from a Laurent
polynomial, a distinguished formal deformation of the associated
non-isolated Gorenstein toric pair and proves that its general fibre is
smooth exactly when the polynomial is $0$-mutable; this does not classify
all components of every non-isolated germ.

\section{Enumeration over the Kreuzer--Skarke
classification}\label{sec:census}

\subsection{Methodology}\label{subsec:method}
The both-sides unit condition is a condition on faces: all edges of
lattice length one
and, by Lemma~\ref{lem:mult}, all dual edges of length one.  We tested it
over
the per-vertex-count copy of the Kreuzer--Skarke classification used in
\cite[\S6.1]{paper1}, with the production implementation built for that
scan.  This covers the complete classification.  All facet incidences, lattice lengths,
determinants, and final predicates are checked in exact integer arithmetic,
with no heuristic pruning.  Planar vertices are ordered by an exact
half-plane and cross-product comparator, so this step also uses no
floating-point arithmetic.  For reproducibility, the ancillary manifest
pins the database repository at revision
\texttt{60c0e119a036} (with the full identifier in the manifest) and gives
the byte size and SHA-256 digest of every input file.  The resumable driver
verifies those digests and validates the row count, result schema, and every
recorded positive hit before accepting an existing result; new results are
written atomically, and every invocation records a transcript.  These
controls fix the identity and integrity of the finite input and saved
artifacts; they do not independently recheck every negative row.
The legacy results predate input-digest recording, but their filenames and
row counts agree with the verified inventory, and the upstream history
identifies the same data objects throughout the scan period.

A separate PALP-based computation, described in \cite[\S6.1]{paper1},
verifies dimension four, reflexivity, actual vertex and facet counts, and
pairwise distinct normal forms under $\mathrm{GL}(4,\ZZ)$ for all
$473{,}800{,}776$ entries.  The published total
\cite{KreuzerSkarke00} then establishes completeness.  This computation
includes the $36$-vertex file, which lay outside the original scan range
and was checked separately for the both-sides condition.  The ancillary
materials contain its checker, full report and input-provenance record.
The singularity classification is the result of the original scan;
the database verification does not repeat its face tests.
Implementation correctness is controlled in
three complementary ways.
(i) Before scanning each file of the database, the implementation was
re-checked against an
independently written reference implementation on random samples from that
file.
(ii) The $3{,}905{,}789$ polytopes with at most nine vertices were scanned
in
full by both implementations, with identical results.  (iii) The $590$
polytopes found pass a global consistency check under polar duality: the
both-sides condition
is self-dual (Corollary~\ref{cor:introboth}), and the polar of every
polytope found,
computed exactly, is again both-sides unit and agrees with one of the
$590$ in vertex
count, facet count and complete face inventory, an invariance that
would generically fail for a false
negative whose polar partner was detected.  An invariant collision
(another member of the $590$ sharing its vertex count, facet count and
face inventory) could in principle mask a miss, so this control is
supporting evidence beside (i) and (ii), not an independent proof.
The separately checked $36$-vertex hexagon product
(see \cite[\S6.1]{paper1}) is included in the classification.
Every polytope passing the scan is then re-verified, and its faces
classified, by the reference implementation; the
ancillary scripts reproduce the full scan from the public database.  The
scan results are included as ancillary data files, and the program
\texttt{both\_sides\_census.py} re-derives every statement of
Theorems~\ref{thm:class}, \ref{thm:dichot} and~\ref{thm:pair} from them.

\subsection{The classification}
Tables~\ref{tab:dist} and~\ref{tab:occ} give the vertex distribution and
face inventory.
\begin{table}[ht]
\centering
\begin{tabular}{lrrrrrrrrrrrrr}
\toprule
vertices & 5 & 6 & 8 & 9 & 10 & 12 & 15 & 16 & 17 & 18 & 19 & 20 & 21\\
\midrule
\# & 2 & 1 & 1 & 1 & 1 & 1 & 1 & 4 & 4 & 7 & 12 & 24 & 37\\
\midrule
vertices & 22 & 23 & 24 & 25 & 26 & 27 & 28 & 29 & 30 & 31 & 33 & 36 &\\
\midrule
\# & 69 & 84 & 98 & 85 & 64 & 42 & 24 & 13 & 9 & 4 & 1 & 1 &\\
\bottomrule
\end{tabular}
\medskip
\caption{The $590$ both-sides unit reflexive $4$-polytopes by vertex count
(no others occur; in particular none with $7$, $11$, $13$, $14$, $32$, $34$
or $35$ vertices).  The population is closed under polar duality
(Corollary~\textup{\ref{cor:introboth}}) and peaks at $24$
vertices.}\label{tab:dist}
\end{table}

\begin{table}[ht]
\begin{center}
\begin{tabular}{llrl}
\toprule
face $[k,i]$ & kind & count & germ of $X$ over it \\
\midrule
$[3,0]$ & triangle & $25{,}873$ & smooth point \\
$[3,1]$ & triangle & $6$ & $\frac13(1,1,1)$ \textbf{(R)} \\
$[3,2]$ & triangle & $10$ & $\frac15(1,1,3)$ \textbf{(R)} \\
$[4,0]$ & zonotope & $15{,}652$ & ordinary double point \\
$[4,1]$ & zonotope & $99$ & cone over $\PP^1{\times}\PP^1$ \\
$[4,1]$ & \emph{asymmetric} & $1$ & cone over $\FF_1$ \textbf{(R)} \\
$[4,2]$ & zonotope & $9$ & type (S) \\
$[5,1]$ & \emph{asymmetric} & $87$ & cone over $\mathrm{dP}_7$ \\
$[6,1]$ & zonotope & $395$ & cone over $\mathrm{dP}_6$, $2$ smoothing components \\
$[6,2]$ & zonotope & $2$ & type (S) \\
$[6,7]$ & zonotope & $3$ & type (S) \\
$[8,4]$ & zonotope & $1$ & type (S), $3$ smoothing components \\
\bottomrule
\end{tabular}
\end{center}
\medskip
\caption{Every two-dimensional face of every both-sides unit polytope, by
class ($k$ = edges, $i$ = interior
points).  Type~\textup{(R)} means that the reduced miniversal base is a
point, while type~\textup{(S)} means that the singular germ admits a
smoothing.  ``Asymmetric'' = neither a
triangle nor centrally symmetric; only the two non-centrally-symmetric
reflexive polygons occur, and each germ appears on $X$ exactly once per
face since all dual edges have length one.  Type-\textup{(S)} germs have
one smoothing component except where noted.}\label{tab:occ}
\end{table}

\begin{proof}[Proof of Theorems~\textup{\ref{thm:class}}
and~\textup{\ref{thm:dichot}}]
Both theorems are exhaustive verifications, summarized in
Tables~\ref{tab:dist} and~\ref{tab:occ}; the completeness of the
underlying enumeration is the content of \S\ref{subsec:method}.
Granting it, the scan of all $473{,}800{,}776$ polytopes yielded
$590$ both-sides unit ones; exact re-analysis of each classifies all their
two-dimensional faces into the twelve classes of Table~\ref{tab:occ}, each
of which is a triangle, a zonotope, or one of the five reflexive polygons
of the smooth toric del Pezzo surfaces.  Explicit $GL_2(\ZZ)$
equivalences identify the $[4,1]$ zonotope with the
$\PP^1{\times}\PP^1$ diamond, the $[6,1]$ zonotope with the
$\mathrm{dP}_6$ hexagon, and the $[3,1]$ triangle with the $\PP^2$
triangle.  The asymmetric $[5,1]$ and $[4,1]$ classes are the
$\mathrm{dP}_7$ pentagon and the $\FF_1$ quadrilateral.  The non-smoothable classes occurring are the
rigid triangles $[3,1]$, $[3,2]$ and the $\FF_1$ quadrilateral; the
$[3,2]$ faces are the ten faces of the quotient simplex of
Example~\ref{ex:pair}, the $[3,1]$ faces lie on the $9$-vertex polytope of
Example~\ref{ex:pair}, and
the $\FF_1$ face on the unique polytope of
Theorem~\ref{thm:pairbody}.  No face of type \textup{(D)} occurs.
\end{proof}

\begin{proof}[Proof of Corollary~\textup{\ref{cor:defonly}}]
A type-\textup{(D)} point of $X$ is the cone over a type-\textup{(D)} face
$F$.  By Theorem~\ref{thm:dichot}, $\Delta$ is then not both-sides unit, so
some face $F'$ has $\ell(F'^\circ) \ge 2$ and Proposition~\ref{prop:corr}(2)
gives a singular curve on $X^\circ$.
\end{proof}

\subsection{The first asymmetric examples}\label{sec:refuted}
Table~\ref{tab:occ} shows that
$88$ asymmetric faces occur, on $55$ polytopes.
The smallest example has $19$ vertices:
\[
\begin{gathered}
\Delta_{19}=\conv\{(1,0,0,0),\,(0,1,0,0),\,(0,0,1,0),\,(0,0,-1,0),\,
(0,-1,0,0),\\
(1,-1,-1,0),\,(0,0,0,1),\,(-1,1,1,-1),\,(0,-1,0,1),\,(0,0,-1,1),\\
(-1,1,0,-1),\,(-1,0,1,-1),\,(-1,0,1,0),\,(-1,1,0,0),\,(0,-1,-1,0),\\
(0,-1,-1,1),\,(-1,0,0,-1),\,(-2,1,1,-1),\,(-1,0,0,1)\},
\end{gathered}
\]
a both-sides unit polytope one of whose faces is the $\mathrm{dP}_7$
pentagon; its hypersurface is a Calabi--Yau threefold with $14$ nodes and
one $\mathrm{dP}_7$-cone point, all with isolated-singular mirror.  That face
is
\[
P = \conv\bigl\{(0,-1,-1,0),\ (0,-1,-1,1),\ (-1,0,0,-1),\
(-2,1,1,-1),\ (-1,0,0,1)\bigr\},
\]
the fifteenth to nineteenth entries of the list, lying in the two facets with
inner normals $u_1 = (1,0,1,0)$ and $u_2 = (1,1,0,0)$; their difference
$(0,-1,1,0)$ is primitive, so $\ell(P^\circ) = 1$.  In the induced lattice
its edge multiset is $\{(0,-1),\,(1,-1),\,(1,0),\,(-1,2),\,(-1,0)\}$, carried
by
$M=\left(\begin{smallmatrix}-1&0\\1&1\end{smallmatrix}\right) \in GL_2(\ZZ)$
onto $\{(0,-1),\,(1,-1),\,(1,1),\,(-1,1),\,(-1,0)\}$, the
$\mathrm{dP}_7$ pentagon of \cite[Rem.~2.6]{paper1}: the unique
unit-edge class with
five vertices and one interior lattice point.

An earlier conjecture excluded asymmetric faces by a proposed local
condition: an asymmetric face with dual-edge length one, all faces of
whose star have unit edges, would force another face in its star to
have dual-edge length at least two.  The condition holds in the
$38{,}571$ candidate configurations among the $1{,}037{,}834$
polytopes with at most eight vertices, but $\Delta_{19}$ disproves it.
Theorem~\ref{thm:class} gives the complete face inventory, whose
geometric explanation is the subject of Question~\ref{q:why}.

\begin{remark}[why not ``simplicial'']\label{rem:selfdual}
The stronger guess that a both-sides unit polytope is simplicial fails for a
structural reason: the condition is self-dual, and ``simplicial and simple''
forces a simplex, contradicting the existence of both-sides unit polytopes
with more than five vertices.  The non-triangular faces are genuine; the
content of Theorem~\ref{thm:class} is that they are zonotopes or reflexive.
\end{remark}

\begin{example}[the low-vertex stratum]\label{ex:pair}
Among the $3{,}905{,}789$ polytopes with at most nine vertices exactly five
are both-sides unit; four of them form two polar pairs.  One pair is the
simplex with smooth anticanonical hypersurface and the simplex
\[
\conv\{(1,0,0,0),\,(0,1,0,0),\,(2,4,5,0),\,(3,3,0,5),\,(-6,-8,-5,-5)\},
\]
whose ten
faces are the $[3,2]$ rigid triangles of Table~\ref{tab:occ} (quotient
points of type $\frac15(1,1,3)$) and whose polar gives a smooth
hypersurface.  The ambient face-fan variety of that polar is singular
at its torus-fixed points, which the generic hypersurface avoids.  A second pair consists of a
simplicial six-vertex polytope with all faces smooth and its nine-vertex
polar, carrying the six $\frac13(1,1,1)$ triangles and nine $[4,2]$
zonotopes.  In each pair, the non-smoothable member mirrors a \emph{smooth}
threefold.  The fifth polytope has eight vertices and sixteen facets, and
all thirty-two of its faces are smooth triangles, so its hypersurface $X$
is smooth.  It shows that closedness under polar duality
(Corollary~\ref{cor:introboth}) does not respect bounds on the number of
vertices: its
polar is a sixteen-vertex both-sides unit polytope all twenty-four of whose
faces are $[4,1]$ diamonds, so the mirror $X^\circ$ carries twenty-four
cone-over-$\PP^1{\times}\PP^1$ points.
\end{example}

\begin{example}[the top of the classification]\label{ex:pair36}
The $36$-vertex product $H \times H$ of two reflexive hexagons is the
reflexive $4$-polytope with the most vertices.  Its database file lies
outside the original scan range of \cite{paper1}, and it was checked
separately.  It is
both-sides unit; its faces are $36$ unit squares and $12$ $\mathrm{dP}_6$
hexagons, so its hypersurface has $48$ singular points, all of type
\textup{(S)}, while its polar, the $12$-vertex free sum
$H^\circ \oplus H^\circ$, has $72$ smooth triangular faces and smooth
hypersurface.
\end{example}

\begin{example}[the smooth mirror pair]\label{ex:24cell}
Exactly eight of the $590$ polytopes have every two-dimensional face
smooth, so that $X$ itself is smooth.  For exactly one of them is the polar
of the same kind: the $24$-cell, self-dual, with $24$ octahedral facets,
$96$ unimodular triangular faces, and no lattice points besides the origin
and its $24$ vertices.  Its generic anticanonical hypersurface is the
classical self-mirror Calabi--Yau threefold with
$(h^{1,1},h^{2,1}) = (20,20)$ and $\chi = 0$; the $(1,1)$-threefolds of
Braun \cite{Braun12} are free quotients of special symmetric members of
this family.  The scan shows uniqueness:
among all $473{,}800{,}776$ reflexive $4$-polytopes, the $24$-cell is the
only one whose generic anticanonical hypersurface and mirror are
\emph{both} smooth.
\end{example}

\section{The unique \texorpdfstring{$\FF_1$}{F1} mirror pair}\label{sec:pair}

\begin{theorem}\label{thm:pairbody}
Let
\[
\begin{gathered}
\Delta_{\FF_1}=\conv\{(1,0,0,0),\,(0,1,0,0),\,(1,-1,0,0),\,(0,0,1,0),\,
(0,0,0,1),\,(0,0,1,-1),\\
(0,0,-1,1),\,(0,0,0,-1),\,(0,0,-1,0),\,(-1,1,0,0),\,(0,-1,0,0),\,
(-1,1,-1,1),\\
(0,-1,0,-1),\,(0,-1,-1,0),\,(-1,1,-1,0),\,(-1,0,0,-1),\,(-1,0,-1,1),\\
(-1,-1,0,-1),\,(-2,1,-1,1),\,(-2,1,-1,0),\,(-1,-1,-1,0),\,(-2,0,-1,0)\}
\subset \ZZ^4 .
\end{gathered}
\]
Then $\Delta_{\FF_1}$ is a both-sides unit reflexive polytope with $22$
vertices and $26$ facets, and:
\begin{enumerate}
\item its $72$ two-dimensional faces are $50$ smooth triangles, $20$ unit
squares, one $\mathrm{dP}_6$ hexagon, and one quadrilateral
$GL_2(\ZZ)$-equivalent to the $\FF_1$ polygon; every dual edge has lattice
length one;
\item the generic anticanonical $X \subset \PP_{\Delta_{\FF_1}}$ is a
Calabi--Yau threefold whose singular locus consists of $20$ ordinary double
points, one $\mathrm{dP}_6$-cone point, and one point with germ the
anticanonical cone over $\FF_1$; since one germ without a smoothing
component obstructs every global smoothing (\S\ref{sec:prelim}), $X$
admits no smoothing \cite[Cor.~4.3]{paper1};
\item the polar $\Delta_{\FF_1}^\circ$ ($26$ vertices, $22$ facets) is
both-sides unit with $68$ two-dimensional faces: $38$ smooth triangles, $26$ unit squares,
two $\mathrm{dP}_7$ pentagons and two $\mathrm{dP}_6$ hexagons; the mirror
$X^\circ$ has $26$ ordinary double points, two $\mathrm{dP}_7$-cone points
and two $\mathrm{dP}_6$-cone points, each with a local smoothing;
\item the maximal projective crepant partial resolutions satisfy
$(h^{1,1},h^{2,1})(\widehat X) = (20,26)$ and
$(h^{1,1},h^{2,1})(\widehat{X^\circ}) = (26,20)$;
\item $\Delta_{\FF_1}$ is the \emph{only}
both-sides unit reflexive
$4$-polytope, up to lattice isomorphism, carrying a
non-smoothable face that is not a triangle.
\end{enumerate}
\end{theorem}

\begin{proof}
(1)--(4) are finite checks on the explicitly listed polytope, a face
enumeration reproducible in any lattice-polytope package, carried out
in exact arithmetic by the ancillary files.\footnote{Ancillary files
\texttt{both\_sides\_census.py} and \texttt{face\_data.py}.}  We record
the $\FF_1$ face itself, since it is the object the statement is about.  Its
four vertices are the fourth, eighteenth, nineteenth and twenty-second entries
of the list above,
\[
F \;=\; \conv\bigl\{(0,0,1,0),\ (-1,-1,0,-1),\ (-2,1,-1,1),\
(-2,0,-1,0)\bigr\},
\]
and the two facets of $\Delta_{\FF_1}$ containing $F$ have inner normals
\[
u_1 = (1,1,-1,-1), \qquad u_2 = (1,0,-1,0),
\]
in the normalization of \S\ref{sec:intro}: both take the value $-1$ at each
of the four vertices of $F$ and are $\ge -1$ on $\Delta_{\FF_1}$.  Their
difference $u_1 - u_2 = (0,1,0,-1)$ is primitive, so
$\ell(F^\circ) = 1$.  In the induced lattice on the affine span of $F$ the
edge multiset is
$\{(0,-1),\,(1,-1),\,(1,1),\,(-2,1)\}$, and
\[
M=\begin{pmatrix}-1&-1\\-1&0\end{pmatrix} \in GL_2(\ZZ),
\qquad \det M = -1,
\]
carries it onto $\{(-2,-1),\,(0,-1),\,(1,0),\,(1,2)\}$, the edge
multiset of the $\FF_1$ polygon of \cite[Rem.~2.6]{paper1}; so $C(F)$ is the
anticanonical cone over $\FF_1$.

The $\mathrm{dP}_6$ hexagon of (1) is recorded the same way.  It is the face
\[
\begin{gathered}
G = \conv\bigl\{(0,0,-1,0),\ (0,-1,-1,0),\ (-1,1,-1,0),\\
(-2,1,-1,0),\ (-1,-1,-1,0),\ (-2,0,-1,0)\bigr\}
\end{gathered}
\]
(the ninth, fourteenth, fifteenth, twentieth, twenty-first and
twenty-second entries of the list), lying in the two facets with inner
normals $u_1 = (0,0,1,0)$ and $u_2 = (0,0,1,1)$, whose difference
$(0,0,0,-1)$ is primitive, so $\ell(G^\circ) = 1$.  Its edge multiset in the
induced lattice is already the standard hexagon,
\[
E(G) = \{(0,-1),\,(1,-1),\,(1,0),\,(0,1),\,(-1,1),\,(-1,0)\},
\]
so no change of basis is needed: $C(G)$ is the anticanonical cone over
$\mathrm{dP}_6$, the class with two smoothing components in the catalogue of
\cite[Rem.~2.6]{paper1}.  The Hodge numbers are
given by Batyrev's formulas \cite[\S4]{Batyrev94}, evaluated in exact
arithmetic.  The mirror swap in (4) is forced by Batyrev duality and serves as
a consistency check.  Assertion (5) is
Theorem~\ref{thm:dichot}: the scan finds no other non-triangle
non-smoothable face on any both-sides unit polytope.
\end{proof}

\begin{remark}[a threefold with no smoothing whose mirror has only
locally smoothable singularities]\label{rem:frozen}
Theorem~\ref{thm:pairbody} pairs a Calabi--Yau threefold that cannot
be smoothed with a mirror all of whose singularities are locally
smoothable.  All $26+2+2$ singular points of the mirror lie over faces
with $\ell(F^\circ)=1$, beyond the single-face construction of
\cite{paper2}.  The node subsystem has no relation with every
coefficient nonzero (Remark~\ref{rem:node-subsystem}).
\end{remark}

\begin{remark}[the node subsystem]\label{rem:node-subsystem}\label{q:mirror-smooth}
Let $q_1,\dots,q_{26}$ be the vertices of
$\Delta_{\FF_1}^\circ$.  For each of its $26$ unit-square faces
$S=(q_a,q_b,q_c,q_d)$ in cyclic order, form the diagonal-relation vector
\[
 \rho_S=e_a-e_b+e_c-e_d\in
 \mathbb{Q}^{\operatorname{Vert}(\Delta_{\FF_1}^\circ)}.
\]
Exact row reduction of the resulting $26\times26$ matrix gives rank $18$.
Exactly two rows are coloops.  The corresponding square faces, displayed
in cyclic order, are
\[
\begin{aligned}
S_-={}&\conv\{(-1,-1,0,-1),\ (-1,0,0,-1),\
                 (0,1,0,-1),\ (0,0,0,-1)\},\\
S_+={}&\conv\{(-1,-1,0,1),\ (0,-1,0,1),\
                 (0,0,0,1),\ (-1,0,0,1)\}.
\end{aligned}
\]
Deleting either row lowers the rank to $17$, deleting both lowers it to
$16$, and deleting any other row leaves rank $18$.
Thus every linear dependence among the $26$ node rows has coefficient zero
on both $\rho_{S_-}$ and $\rho_{S_+}$; in particular there is no dependence
with all coefficients non-zero.  The calculation is exact and is reproduced
by the ancillary row-reduction script.
\footnote{\texttt{paper3\_node\_relations.py}.}

For a hypothetical all-nodal threefold this fails the
Friedman--Batyrev--Kreuzer condition recalled in \cite[\S7]{paper2}.
It does \emph{not} decide the actual $X^\circ$: its four additional
del Pezzo-cone points have divisorial crepant resolutions, and their local
smoothing directions may contribute to the mixed local-to-global
obstruction map.  The mixed necessity theorem of \cite[Theorem~6.1]{paper4} requires a
relation among the classes in $H_2(L_p;\mathbb Q)$ of the singularity
links whose image in $H_2(\widehat X^\circ;\mathbb Q)$ is zero, whose
cone-point components belong to the kernels
$\ker(H_2(L_p;\mathbb Q)\to H_2(M_p;\mathbb Q))$ for the chosen local
Milnor fibres, and whose node coefficients are all nonzero.  The
resolution $\widehat X^\circ$ uses small resolutions at the nodes and
blows up the cone points.  In this example the two coloop coefficients
vanish on every such relation, for either smoothing component at each
degree-six point \cite[\S5.3, Theorem~6.5]{paper4}.
Hence $X^\circ$ is nonsmoothable.  Both members of the pair
are therefore nonsmoothable, although every singularity of $X^\circ$
is locally smoothable.
\end{remark}

\section{Questions}\label{sec:questions}

\begin{question}\label{q:why}
Can one characterize the unit-edge lattice polygons that occur as
two-dimensional faces of reflexive polytopes whose edges and polar edges
all have lattice length one?  In dimension four,
Theorem~\ref{thm:class} gives the twelve classes in
Table~\ref{tab:occ}.  A geometric
explanation would account for the occurrence of only triangles,
zonotopes, and reflexive polygons without using the classification.
The local condition discussed in \S\ref{sec:refuted} is insufficient.
\end{question}

\begin{question}\label{q:mutable}
Corti--Filip--Petracci conjecturally relate smoothing components of $C(Q)$ to $0$-mutable
Laurent polynomials with Newton polygon $Q$ \cite{CFP}, and Filip
\cite{Filip25} constructs distinguished formal deformations in the
non-isolated case, detecting when their general fibres are smooth.  In that language, what
distinguishes the twelve face classes of Table~\ref{tab:occ} --- and what is
the mirror-theoretic meaning of the unique $\FF_1$ face?  One would like to
read Theorem~\ref{thm:dichot} off the canonical scattering diagram of the
mirror.
\end{question}

\bigskip
\noindent{\sc Bernd Johannes Wuebben, New York, NY,} \texttt{wuebben@gmail.com}

\end{document}